\documentclass[english,12pt,a4paper]{amsart}

\usepackage{amssymb}
\usepackage{hyperref}

\theoremstyle{plain}
\newtheorem{Theorem}{Theorem}
\newtheorem{Lemma}[Theorem]{Lemma}
\newtheorem{Corollary}[Theorem]{Corollary}
\newtheorem{Example}[Theorem]{Example}

\newcommand{\ddim}{{\rm dim}}
\newcommand{\trdeg}{{\rm tr.deg}}
\newcommand{\ch}{{\rm char}}
\newcommand{\setmin}{\smallsetminus}

\begin{document}

\bibliographystyle{alpha}

\title[Subfields of ample fields I]{Subfields of ample fields I. \\ Rational maps and definability}
\author{Arno Fehm}

\begin{abstract}
Pop proved that a smooth curve $C$ over an ample field $K$
with $C(K)\neq\emptyset$ has $|K|$ many rational points.
We strengthen this result 
by showing that there are $|K|$ many rational points
that do not lie in a given proper subfield,
even after applying a rational map.
As a consequence we
gain insight into the structure of existentially definable subsets
of ample fields. 
In particular, we prove that a perfect ample field
has no existentially definable proper infinite subfields.
\end{abstract}

\maketitle

\section*{Introduction}

Recall that a field $K$ is called {\em ample} (or {\em large} after \cite{Pop}) 
if every smooth curve $C$ defined over $K$ with a
$K$-rational point has infinitely many of them.
The class of ample fields subsumes several seemingly unrelated
classes of fields -- separably closed fields and pseudo
algebraically closed fields, real closed fields, and Henselian
valued fields.

By introducing ample fields, Pop was able to reprove and generalise a couple of deep
results on a certain important Galois theoretic conjecture of D\`ebes and Deschamps 
(a vast generalization of the classical inverse Galois problem, see \cite[\S2.1.2]{DD}).
Furthermore, all fields for which this conjecture has been proven so far turned out to be ample.
Therefore, the class of ample fields now plays a decisive role in the study of general Galois theory. 
Moreover, in recent years ample fields drew attention in several other branches
of mathematics -- for example in the study of rationally connected varieties and torsors, 
see \cite{Kollar}, \cite{MoretBailly},
in the study of Abelian varieties, see \cite{Kobayashi}, \cite{LozanoRobledo}, \cite{FehmPetersen},
and in the study of definability in fields, see \cite{Koenigsmann2}, \cite{PoonenPopPublished}, \cite{JunkerKoenigsmann}.
We refer the reader to the introduction of \cite{PopHenselian} for
a more extensive survey of ample fields.

Harbater and Stevenson called a field $K$ {\em very large}
if every smooth curve $C$ defined over $K$ with a
$K$-rational point has $|K|$ many such points. 
They proved that the absolute Galois group of a function field of one variable
over a very large field $K$ is a so-called quasi-free group,
and they asked if the same holds for every ample field $K$ (see \cite{HarbaterStevenson}).
Pop gave a positive answer to this question,
by showing that every ample field is actually very large, cf.~\cite[Proposition 3.3]{Harbater}.

In this work, we considerably strengthen this result by proving
that in the situation above, $C$ has $|K|$ many rational points {\em with respect to any proper subfield} of $K$. 
That is:

\begin{Theorem}\label{main}
Let $K$ be an ample field,
$C$ a curve defined over $K$ with a simple $K$-rational point and
$\varphi:C\rightarrow C^\prime$ a separable dominant $K$-rational map to an affine
curve $C^\prime\subseteq\mathbb{A}^n$ defined over $K$.
Then for every proper subfield $K_0$ of $K$,
$$
 |\varphi(C(K))\setmin\mathbb{A}^n(K_0)|=|K|.
$$
\end{Theorem}

The proof of Theorem \ref{main} is carried out in Section~\ref{sec1}.
It is based on a careful analysis of the trick of Koenigsmann that is used in Pop's proof.

Section~\ref{sec2} gives a model theoretic reinterpretation of Theorem~\ref{main}:
Motivated by what is known about definable subsets of some well-studied 'classical' ample fields,
we apply Theorem~\ref{main} to investigate the structure of existentially definable (i.e.~diophantine) subsets of ample fields.
For example we prove that a perfect ample field has no existentially definable proper infinite subfields,
and we point out connections with the recent work \cite{JunkerKoenigsmann}.

\section*{Notation}

All varieties are geometrically irreducible and geometrically reduced.
For any set $X$, the cardinality of $X$ is denoted by $|X|$.
The algebraic closure of a field $K$ is denoted by $\tilde{K}$.
Formulas are first order formulas in the language of rings.

\section{Varieties over subfields of ample fields}\label{sec1}

We start with a well known lemma from linear algebra:

\begin{Lemma}\label{LemmaLinAlg}
Let $V$ be a vector space over an infinite field $K$.
Suppose $V=\bigcup_{i\in I}W_i$ is a union
of proper linear subspaces $W_i$ of $V$
with
$\sup_{i\in I}\ddim_K(W_i)<\infty$.
Then $|I|\geq|K|$.
\end{Lemma}

\begin{proof}
If $\ddim_K(V)=\infty$, replace $V$ by any subspace $V^\prime\subseteq V$
with $\sup_{i\in I}\ddim_K(W_i)<\ddim_K(V^\prime)<\infty$,
and $W_i$ by $W_i\cap V^\prime$ to assume without loss of generality
that $\ddim_K(V)<\infty$.
Then proceed by induction on $\ddim_K(V)$:

The case $\ddim_K(V)=1$ clearly cannot occur.
If $\ddim_K(V)>1$ and $|I|<|K|$,
then there is a subspace $W$ of $V$ of codimension $1$ such that $W\neq W_i$ for all $i$.
But then $W=\bigcup_{i\in I}(W_i\cap W)$ and each $W_i\cap W$ is a proper subspace of $W$,
contradicting the induction hypothesis.
\end{proof}

The next lemma is a refinement of the method used in \cite{Koenigsmann2}.

\begin{Lemma}\label{KoenigsmannLemma}
Let $K$ be an ample field, $f\in K[X,Y]$ an absolutely irreducible polynomial,
and $y_0\in K$ such that
$f(0,y_0)=0$ and $\frac{\partial f}{\partial Y}(0,y_0)\neq 0$.
Denote by $C$ the affine curve $f(x,y)=0$,
and by $\pi$ the projection $(x,y)\mapsto x$.
Then
for every proper infinite subfield $K_0$ of $K$,
$$
 |\pi (C(K))\setmin K_0|\geq |K_0|.
$$
\end{Lemma}

\begin{proof}
Let $c\in K$.
Define $x_1^\prime=ct,x_2^\prime=t\in K((t))$ and $g_i(Y)=f(x_i^\prime,Y)\in K[[t]][Y]$, $i=1,2$.
Then $g_i$ modulo $t$ has the simple zero $y_0$.
By Hensel's lemma, there are $y_1^\prime,y_2^\prime\in K((t))$ such that
$f(x_1^\prime,y_1^\prime)=f(x_2^\prime,y_2^\prime)=0$.
Since $K$ is ample, it is existentially closed in $K((t))$, cf.~\cite[Proposition 1.1]{Pop}.
Therefore we can find $x_1,x_2,y_1,y_2\in K$ such that
$f(x_1,y_1)=f(x_2,y_2)=0$,
$x_2\neq 0$,
and $c=x_1/x_2$.
Since $c$ was arbitrary, we have shown:
$$
 K = \left\{ \frac{x_1}{x_2} : (x_1,y_1),(x_2,y_2)\in C(K), x_2\neq 0 \right\}.
$$

For simplicity, write
$A=\{ P \in C(K) : \pi P\in K_0\}$,
$B=\{ P \in C(K) : \pi P\notin K_0\}$,
$A^\prime = A\setmin\pi^{-1}(0)$.
Then $C(K)=A\cup B$ and $C(K)\setmin\pi^{-1}(0)=A^\prime\cup B$,
so

\begin{eqnarray*}
 K &=& \bigcup_{P\in C(K)\atop Q\in C(K)\setmin\pi^{-1}(0)}\{\frac{\pi P}{\pi Q}\}\\
   &=& \bigcup_{P\in A\atop Q\in A^\prime}\{\frac{\pi P}{\pi Q}\}\cup
     \bigcup_{P\in A \atop Q\in B}\{\frac{\pi P}{\pi Q}\}\cup
     \bigcup_{P\in B \atop Q\in A^\prime}\{\frac{\pi P}{\pi Q}\}\cup
     \bigcup_{P\in B \atop Q\in B}\{\frac{\pi P}{\pi Q}\}\\
   &\subseteq& K_0\cup\bigcup_{Q\in B} K_0\cdot \frac{1}{\pi Q} \cup
                           \bigcup_{P\in B} K_0\cdot \pi P \cup
                  \bigcup_{P\in B \atop Q\in B}K_0\cdot\frac{\pi P}{\pi Q}
\end{eqnarray*}

Thus, $K$ is covered by at most $1+2\cdot|B|+|B|^2$  $K_0$-subspaces of dimension $1$.
Since $[K:K_0]\geq 2$, these are proper subspaces and Lemma~\ref{LemmaLinAlg} implies that
$1+2\cdot|B|+|B|^2\geq|K_0|$, so
$|B|\geq |K_0|$.
Since the fibers of $\pi|_C$ are finite, this implies
$|\pi (C(K))\setmin K_0|=|B|\geq |K_0|$.
\end{proof}

\begin{proof}[Proof of Theorem \ref{main}]
The characterization of separating elements in separable function fields of one variable shows
that we can replace $\varphi$ by one of its coordinate functions and assume that
$C^\prime=\mathbb{A}^1$, i.e. $\varphi$ is a separable rational function on $C$.
Then it suffices to prove that $|\varphi(C(K))\setmin K_0|=|K|$.

Since $K$ is ample and $C$ has a simple $K$-rational point, $C(K)$ is infinite.
If $K_0$ is finite and $K$ algebraic over $K_0$, then,
since $\varphi$ has finite fibers,
$|\varphi(C(K))\setmin K_0|=\aleph_0=|K|$.
So assume without loss of generality that $|K_0|=\aleph_0=|K|$ or $K$ is transcendental over $K_0$.
Since a purely transcendental extension of a field is never ample,
we may adjoin a transcendental basis of $K|K_0$ to $K_0$ and assume that $|K_0|=|K|$.

Since $\varphi$ is separable,
there is a plane affine curve $D\subseteq\mathbb{A}^2$ defined by an absolutely irreducible polynomial $f\in K[X,Y]$,
separable in $Y$,
and a $K$-birational map $\eta:C\rightarrow D$ such that, with $\pi$ the projection on the first coordinate,
$\pi\circ\eta=\varphi$.
Since $\eta$ is birational, there are cofinite subsets $C_0$ of $C$ and $D_0$ of $D$ such that
$\varphi$ is defined on $C_0$ and
$\eta$ maps $C_0$ bijectively onto $D_0$ with inverse $\eta^{-1}$.
In particular, $\eta(C_0(K))=D_0(K)$ and $\varphi(C_0(K))=\pi(D_0(K))$.

Since $C(K)$ is infinite, also $D(K)$ is infinite.
Choose $P_1=(x_1,y_1)\in D(K)$ with $\frac{\partial f}{\partial Y}(x_1,y_1)\neq 0$ and
denote the curve defined by $f_1(X,Y)=f(X+x_1,Y)\in K[X,Y]$ by $D_1$.
By Lemma \ref{KoenigsmannLemma} applied to $f_1$, $y_1$, $K_0$  we get that
$|\pi(D_1(K))\setmin K_0|\geq|K_0|=|K|$.
Thus in particular $|\pi (D(K))|=|\pi (D_1(K))|=|K|$.

Assume that $|\pi (D(K))\setmin K_0|<|K|$.
Then there are infinitely many $P=(x,y)\in D(K)$ with $x=\pi P\in K_0$.
Choose $P_2=(x_2,y_2)\in D(K)$ with
$x_2=\pi P_2\in K_0$ and $\frac{\partial f}{\partial Y}(x_2,y_2)\neq 0$ and
denote the curve defined by $f_2(X,Y)=f(X+x_2,Y)\in K[X,Y]$ by $D_2$.
By Lemma \ref{KoenigsmannLemma} applied to $f_2$, $y_2$, $K_0$ we get that
$|\pi(D_2(K))\setmin K_0|\geq|K_0|=|K|$.
But $x_2\in K_0$ implies $|\pi(D_2(K))\setmin K_0|=|\pi(D(K))\setmin K_0|<|K|$, a contradiction.

Thus $|\pi(D(K))\setmin K_0|=|K|$.
But then also $|\varphi(C(K))\setmin K_0|=|K|$,
as was to be shown.
\end{proof}

We now give two immediate consequences of Theorem \ref{main}.

\begin{Corollary}\label{MainWeak}
Let $K$ be an ample field and $V$ a positive dimensional variety defined over a subfield $K^\prime$ of $K$
with a simple $K^\prime$-rational point.
Then for every proper subfield $K_0\supseteq K^\prime$ of $K$,
$|V(K)\setmin V(K_0)|=|K|$.
In particular, $K=K^\prime(V(K))$ and $|V(K)|=|K|$.
\end{Corollary}

\begin{proof}
Choose an affine curve $C$ on $V$ defined over $K^\prime$
which has a simple $K^\prime$-rational point (see \cite{KleimanAltman}).
By Theorem \ref{main}, $|C(K)\setmin C(K_0)|=|K|$, so $|V(K)\setmin V(K_0)|=|K|$.
If $K_1=K^\prime(V(K))$ is a proper subfield of $K$, then
$|V(K)\setmin V(K_1)|=|K|$, contradicting $V(K)=V(K_1)$.
\end{proof}

The very last statement of this corollary
is the result of Pop mentioned in the introduction.
A weak consequence of Corollary \ref{MainWeak} is used in \cite{FehmPetersen}
to study the rank of Abelian varieties over ample fields.

\begin{Corollary}\label{Embedding}
If $E$ is ample and $K\subseteq E$ is a subfield with $\trdeg(E|K)\geq1$, then for every function field of one variable $F|K$
with a rational place,
there is a $K$-embedding of $F$ into $E$.
\end{Corollary}

\begin{proof}
Let $K_0=\tilde{K}\cap E$ and
let $C$ be a model of $F|K$ with a simple $K$-rational point.
By Theorem \ref{main}, $|C(E)\setmin C(K_0)|=|K|$.
Every point $\mathbf{x}\in C(E)\setmin C(K_0)$ is a generic point of $C$ over $K$,
so $F\cong_KK(C)\cong_KK(\mathbf{x})\subseteq E$.
\end{proof}

From Corollary \ref{Embedding} and the Riemann-Hurwitz theorem it follows
immediately that function fields are non-ample
and one can deduce many examples
of non-ample infinite algebraic extensions of function fields.
With a higher dimensional generalization of Lemma~\ref{KoenigsmannLemma},
one can prove a higher dimensional analogue of Corollary~\ref{Embedding}.
We will deal with these aspects in the forthcoming note \cite{subfields2}.

\section{Definability in ample fields}\label{sec2}

For many of the 'classical' ample fields $K$, a lot is known about definable infinite subsets $X\subseteq K$.
For example, if $K$ is  algebraically closed, then $X$ is cofinite.
If $K$ is real closed then $X$ is a finite union of intervals.
If $K$ is $p$-adically closed, then $X$ contains a non-empty open subset\footnote{see \cite{Macintyre}}.
In particular, in each of these cases $K$ has no definable proper infinite subfields.
The same is true for perfect PAC fields\footnote{as follows from \cite[Proposition 5.3]{DriesMacintyre}},
and
a recent result in \cite{JunkerKoenigsmann} shows that
it is also true in the case that $K$ is Henselian of characteristic zero.
We now show that for the huge class of perfect ample fields,
there are at least no {\em existentially} definable proper infinite subfields.

Note that it is {\em not true} that an ample field of characteristic zero
has no definable proper subfields (see Example \ref{power} below),
so our result on {\em existentially} definable subfields is in some sense the best we can expect
for the class of ample fields.
Furthermore note that the corresponding statement  for {\em non-ample} fields is {\em not true} in general:
If $F|K$ is a function field
and $K$ is algebraically closed, say,
then $K$ {\em is} existentially definable in $F$ (see \cite{Duret}, \cite{Koenigsmann2}).

\begin{Lemma}\label{Lemmadefinable}
Let $K$ be a perfect ample field and $K_0\subseteq K$ a subfield.
Let $X\subseteq K$ be existentially $K_0$-definable
such that $X\not\subseteq \tilde{K_0}$.
Then $|X\setmin K_1|=|K|$ for every proper subfield $K_1$ of $K$.
\end{Lemma}

\begin{proof}
Let $\varphi(x_0)$ be an existential formula that defines $X$.
Without loss of generality assume that $\varphi(x_0)$ is of the form
$$
(\exists x_1,\dots,x_n)(\bigwedge_{i=1}^k f_i(x_0,x_1,\dots,x_n)=0),
$$
where $f_1,\dots,f_k\in K_0[X_0,\dots,X_n]$.

Since $X\not\subseteq\tilde{K_0}$,
there are $x_0\in X$ and $x_1,\dots,x_n\in K$
such that with $\mathbf{x}=(x_0,\dots,x_n)$, $\bigwedge_{i=1}^k f_i(\mathbf{x})=0$
and $\trdeg(K_0(x_0)|K_0)=1$.
Let $F=K_0(\mathbf{x})$ and
adjoin a transcendence base of $F|K_0(x_0)$ to $K_0$ to assume without loss of generality
that $\trdeg(F|K_0)=1$.
Then replace $K_0$ by $\tilde{K_0}\cap K$,
so that $K_0$ is perfect and $F|K_0$ is a separable function field of one variable.

If $\ch(K)=p>0$,
let $F^\prime$ be the separable closure of $K_0(x_0)$ in $F$.
Then $F^\prime=F^{p^r}K_0=K_0(x_0,x_1^{p^r},\dots,x_n^{p^r})$ for some $r\geq 0$.
Define $f_i^\prime(X_0,\dots,X_n)=f_i^{p^r}(X_0,X_1^{p^{-r}},\dots,X_n^{p^{-r}})\in K_0[X_0,\dots,X_n]$,
$\varphi^\prime(x_0)$ as $(\exists x_1,\dots,x_n)(\bigwedge_{i=1}^k f_i^\prime(x_0,x_1,\dots,x_n)=0)$,
and $\mathbf{x}^\prime=(x_0,x_1^{p^r},\dots,x_n^{p^r})$.
Then $\bigwedge_{i=1}^k f_i^\prime(\mathbf{x}^\prime)=0$ and $K_0(\mathbf{x}^\prime)=F^\prime$.
Since $K$ is perfect, the subset defined by $\varphi^\prime$ is exactly $X$.
Thus replace $\varphi$ by $\varphi^\prime$, $\mathbf{x}$ by $\mathbf{x}^\prime$ and $F$ by $F^\prime$
to assume without loss of generality that $F|K_0(x_0)$ is separable.

Let $C$ be the locus of $\mathbf{x}$ over $K_0$. That is,
$C\subseteq\mathbb{A}^{n+1}$ is an affine curve defined over $K_0$
with $F\cong_{K_0} K_0(C)$,
and $\mathbf{x}$ is a simple $F$-rational point of $C$.
If $\mathbf{y}\in C(K)$ is another point of $C$,
then there is a $K_0$-specialization $\eta_{\mathbf{y}}:K_0[\mathbf{x}]\rightarrow_{K_0} K_0[\mathbf{y}]$.
In particular, $0=\eta_{\mathbf{y}}(f_i(\mathbf{x}))=f_i(\mathbf{y})$ for all $i$,
i.e. $y_0\in X$.
This means that if $\pi:\mathbb{A}^{n+1}\rightarrow\mathbb{A}^1$ is the projection
on the first coordinate, then $\pi(C(K))\subseteq X$.

But $\pi|_C$ is separable since $F|K_0(x_0)$ is separable,
so $|\pi(C(K))\setmin K_1|=|K|$ by Theorem \ref{main}.
Thus $|X\setmin K_1|\geq |\pi(C(K))\setmin K_1|=|K|$.
\end{proof}

\begin{Corollary}\label{definable}
Let $K$ be a perfect ample field and let $X\subseteq K$
be existentially $K$-definable and infinite.
Then $|X\setmin K_1|=|K|$
for every proper subfield $K_1$ of $K$.
\end{Corollary}

\begin{proof}
Let $\varphi$ be an existential formula that defines $X$,
and let $K_0\subseteq K$ be the field generated by the parameters used in $\varphi$.
Let $K^*\supseteq K$ be an $\aleph_1$-saturated ultrapower of $K$.
Then the subset $X^*\subseteq K^*$ defined by $\varphi$ in $K^*$ is uncountable.
Since $K_0$ is finitely generated,
$X^*\not\subseteq\tilde{K_0}$.
Thus by Lemma \ref{Lemmadefinable},
$|X^*\setmin K_1^*|=|K^*|$
for every proper subfield $K_1^*$ of $K^*$.

Let $K_1$ be any proper subfield of $K$.
Then the corresponding ultrapower $K_1^*$ is a proper subfield of $K^*$,
thus $|X^*\setmin K_1^*|=|K^*|$,
which implies
$|X\setmin K_1|=\infty$.
So if $K$ is countable, $|X\setmin K_1|=\aleph_0=|K|$ as was to be shown.
If $K$ is uncountable, then $\tilde{K_0}\cap K$ is a proper subfield of $K$,
so $|X\setmin (\tilde{K_0}\cap K)|=\infty$.
In particular $X\not\subseteq \tilde{K_0}$.
Thus by Lemma~\ref{Lemmadefinable}, $|X\setmin K_1|=|K|$.
\end{proof}

\begin{Corollary}\label{definable2}
A perfect ample field $K$
has no
existentially $K$-definable proper infinite subfields.\qed
\end{Corollary}

The following example shows that Corollary~\ref{definable2} does not hold for arbitrary $K$-definable subfields
(rather than existentially $K$-definable subfields):

\begin{Example}\label{power}
The prime field $\mathbb{Q}$ is a $\emptyset$-definable proper subfield of the ample field $\mathbb{Q}((X,Y))$.
\end{Example}

\begin{proof}
As noted in \cite{PopHenselian},
$K((X,Y))$ is ample for any field $K$.
But the ring $K[[X,Y]]$ is definable in $K((X,Y))$, see \cite[Theorem 3.34]{JL},
and if ${\rm char}(K)=0$, then $\mathbb{N}$ is definable in $K[[X,Y]]$,
see \cite[Th\'eor\`eme~2.1]{Delon}.
\end{proof}

Note that the results of this section are closely related to \cite{JunkerKoenigsmann}.
There, a field $K$ is called {\em very slim},
if in all fields elementary equivalent to $K$,
'the model theoretic algebraic closure in $K$ coincides
with the relative field theoretic algebraic closure', cf.~\cite[Definition 1.1]{JunkerKoenigsmann}.

Since a very slim field has no definable proper infinite subfields, see \cite[Proposition 4.1]{JunkerKoenigsmann},
Example \ref{power}
gives a positive answer to (1) of \cite[Question~8]{JunkerKoenigsmann}:
Perfect ample fields that are not very slim do exist.
In particular, \cite[Theorem 5.5]{JunkerKoenigsmann},
which states that Henselian fields of characteristic zero are very slim,
does not generalise to the class of ample fields of characteristic zero.

Since a model complete field is perfect and all its definable subsets are existentially definable,
Lemma~\ref{Lemmadefinable}
immediately gives a new proof for
\cite[Theorem~5.4]{JunkerKoenigsmann}:
Every model complete ample field is very slim.

\section*{Acknowledgements}

The author would like to thank Lior Bary-Soroker and Moshe Jarden
for helpful comments and encouragement,
and Elad Paran for contributing to the introduction.
This work was supported by the European Commission under contract MRTN-CT-2006-035495.


\end{document}